\documentclass[11pt,fleqn]{article}

\usepackage{amssymb, amsthm, amsmath}

\numberwithin{equation}{section}

\theoremstyle{plain}
\newtheorem{theorem}{Theorem}[section]
\newtheorem{lemma}{Lemma}[section]

\theoremstyle{definition}
\newtheorem{definition}{Definition}[section]

\theoremstyle{remark}

\title{Determining nodes for semilinear parabolic equations}

\author{Ry\^{o}hei Kakizawa\thanks{Graduate School of Mathematical Sciences, The University of Tokyo, 3-8-1 Komaba Meguro-ku Tokyo 153-8914, Japan (\textit{E-mail address:} kakizawa@ms.u-tokyo.ac.jp)}}

\date{}

\begin{document}

\maketitle

\begin{abstract}
We are concerned with the uniqueness of the asymptotic behavior of strong solutions of the initial-boundary value problem for general semilinear parabolic equations by the asymptotic behavior of these strong solutions on a finite set of an entire domain.
More precisely, if the asymptotic behavior of a strong solution is known on an appropriate finite set, then the asymptotic behavior of a strong solution itself is entirely determined in a domain.
We prove the above property by the energy method.
\end{abstract}


\section{Introduction}
Let $\Omega$ be a bounded domain in $\mathbb{R}^n$ $(n \in \mathbb{Z}, \ n\geq 2)$ with its $C^{0,1}$-boundary $\partial\Omega$, $H$ be a closed subspace of $L^{2}(\Omega)$, $V=H^{1}_{0}(\Omega)\cap H$.
Our problem is the following strong formulation of the initial-boundary value problem for the semilinear parabolic equation:
\begin{equation}
\begin{split}
&d_{t}u+Au+Bu=f & \mathrm{in} \ L^{2}((0,\infty);H), \\
&u(0)=u_{0} & \mathrm{in} \ V,
\end{split}
\end{equation}
where $u$ is a strong solution of (1.1), $A$ is a closed linear operator from $D(A)$ to $H$, $B$ is a nonlinear operator from $D(B)$ to $H$, $f$ is a nonhomogeneous term, $u_{0}$ is an initial data of $u$.
Moreover, $D(A)$ and $D(B)$ are domains of $A$ and $B$ respectively.
A typical example of $(1.1)_{1}$ is the following semilinear heat equation:
\begin{equation*}
\partial_{t}u-k\Delta u-|u|^{p-1}u=0,
\end{equation*}
where $k>0$, $p>1$.
The existence and uniqueness of strong solutions of the initial-boundary value problem for the semilinear heat equation has been much studied for fifty years.
The stationary problem of (1.1) is the following boundary value problem for the semilinear elliptic equation:
\begin{equation}
A\bar{u}+B\bar{u}=\bar{f} \ \mathrm{in} \ H,
\end{equation}
where $\bar{u}$ is a strong solution of (1.2), $\bar{f}$ is a nonhomogeneous term.
It is well known in \cite{Rabinowitz} that the stationary problem for the semilinear heat equation have a trivial solution and nontrivial solutions.
It is one of interesting questions whether a strong solution of (1.1) converges to a trivial or nontrivial solution of (1.2).
The conclusion for asymptotic properties of strong solutions of (1.1) can be given by the theory of determining nodes introduced by Foias and Temam \cite{Foias 2}.
The approach of determining nodes is quite natural from the computational point of view.
In general, strong solutions of the initial-boundary value problem for semilinear parabolic equations is uniquely determined by determining nodes which can be obtained from finite many measurements.
Some problems related to determining nodes for semilinear parabolic equations have been discussed.
It is proved by Foias and Kukavica \cite{Foias 1}, Kukavica \cite{Kukavica} and Oliver and Titi \cite{Oliver} that there exist determining nodes for the Kuramoto-Shivashinsky equation, the complex Ginzbrug-Landau equation and the semilinear Schr\"{o}dinger equation respectively.
In recent years, Lu and Shao \cite{Lu} studied the existence of determining nodes for partly dissipative reaction diffusion systems including the FitzHugh-Nagumo equations.
However, determining nodes for the semilinear heat equation have been not considered yet.
It is necessary to discuss the existence of determining nodes for general semilinear parabolic equations.

In this paper, we are concerned with the determination of the asymptotic behavior of strong solutions of (1.1) by determining nodes.
The theory of determining nodes for not only the Navier-Stokes equations but also the semilinear heat equation can be unified.
One of our main results is stated as follows: There exists a finite set $E$ of determining nodes such that if $u(x,t)-v(x,t)\rightarrow 0$ as $t\rightarrow \infty$ for any $x \in E$, then $u(\cdot,t)-v(\cdot,t)\rightarrow 0$ as $t\rightarrow \infty$ in $V\cap C^{0,\mu}(\overline{\Omega})$ for any $0<\mu<1/2$.
We prove the above results by the argument based on \cite{Foias 2, Lu}.

This paper is organized as follows: In section 2, we define function spaces, basic notation used in this paper and strong solutions of (1.1) and (1.2), and state our main results and some lemmas for them.
We prove our main results in section 3.
Finally, we apply our main results to the semilinear heat equation and the Navier-Stokes equations in section 4.

\section{Preliminaries and main results}
\subsection{Function spaces}
All functions appearing in this paper are either $H$ or $H^{n}$-valued.
For the sake of simplicity, we will not distinguish them from their values in notation.

The norm in $L^{p}(\Omega)$ $(1\leq p\leq\infty)$ and in $H^{m}(\Omega)$ (the Sobolev space, $m \in \mathbb{Z}, \ m\geq 0$) are denoted by $\|\cdot\|_{L^{p}(\Omega)}$ and $\|\cdot\|_{H^{m}(\Omega)}$ respectively, $H^{0}(\Omega)=L^{2}(\Omega)$.
Moreover, the scalar product in $L^{2}(\Omega)$ and in $H^{m}(\Omega)$ are denoted by $(\cdot,\cdot)_{L^{2}(\Omega)}$ and $(\cdot,\cdot)_{H^{m}(\Omega)}$ respectively.
$C^{\infty}_{0}(\Omega)$ is the set of all functions which are infinitely differentiable and have compact support in $\Omega$.
$H^{1}_{0}(\Omega)$ is the completion of $C^{\infty}_{0}(\Omega)$ in $H^{1}(\Omega)$.
$H^{1}_{0}(\Omega)$ is characterized as $H^{1}_{0}(\Omega)=\{ u \in H^{1}(\Omega) \ ; \ u|_{\partial\Omega}=0\}$.
It is well known in the theory of Hilbert spaces that $L^{2}(\Omega)$ is decomposed into $L^{2}(\Omega)=H\oplus H^{\perp}$, where $H^{\perp}$ is the orthogonal complement of $H$.
Let $P$ be the orthogonal projection of $L^{2}(\Omega)$ onto $H$.
The norm in $C(\overline{\Omega})$ is denoted by $\|\cdot\|_{C(\overline{\Omega})}$.
$C^{0,\mu}(\overline{\Omega})$ $(0<\mu\leq 1)$ is the Banach space of all functions which are uniformly H\"{o}lder continuous with the exponent $\mu$ on $\overline{\Omega}$.
The norm in $C^{0,\mu}(\overline{\Omega})$ is denoted by $\|\cdot\|_{C^{0,\mu}(\overline{\Omega})}$, that is,
\begin{equation*}
\|u\|_{C^{0,\mu}(\overline{\Omega})}:=\|u\|_{C(\overline{\Omega})}+[[u]]_{C^{0,\mu}(\overline{\Omega})}, \ [[u]]_{C^{0,\mu}(\overline{\Omega})}:=\sup_{x,y \in \overline\Omega, x \neq y}\frac{|u(x)-u(y)|}{|x-y|^{\mu}}.
\end{equation*}
Let $I$ be an open interval in $\mathbb{R}$, $(X,\|\cdot\|_{X})$ be a Banach space.
$L^{p}(I;X)$ $(1\leq p<\infty)$ is the Banach space of all $X$-valued functions $u$ which $u$ is strongly measurable and $\|u(t)\|^{p}_{X}$ is integrable in $I$.
$L^{\infty}(I;X)$ is the Banach space of all $X$-valued functions $u$ which $u$ is strongly measurable and $\|u(t)\|_{X}$ is essentially bounded in $I$.
The norm in $L^{p}(I;X)$ and in $L^{\infty}(I;X)$ are denoted by $\|\cdot\|_{L^{p}(I;X)}$ and $\|\cdot\|_{L^{\infty}(I;X)}$ respectively.
In the case where $I$ is a bounded closed interval in $\mathbb{R}$, $C(I;X)$ is the Banach space of all $X$-valued functions which are continuous on $I$.
If $I$ is not bounded or closed, $C_{b}(I;X)$ is the Banach space of all $X$-valued functions which are bounded and continuous in $I$.
The norm in $C(I;X)$ and in $C_{b}(I;X)$ is denoted by $\|\cdot\|_{C(I;X)}$ and $\|\cdot\|_{C_{b}(I;X)}$ respectively.

\subsection{Strong solutions of \rm{(1.1)}, \rm{(1.2)}}
Let us define the closed linear operator $A$ and the nonlinear operator $B$ which appeared in (1.1).
$Au=-a\Delta u$ $(a>0)$ is a typical example of $A$, the norm induced by $A$ is equivalent to a norm in $H^{2}(\Omega)\cap H^{1}_{0}(\Omega)$.
It is important for our main results that $Bu=-|u|^{p-1}u$ $(p>1)$ and $Bu=P(u\cdot\nabla)u$ can be considered.
$A$ is a closed linear operator from $D(A):=H^{2}(\Omega)\cap V$ to $H$ defined as
\begin{equation*}
Au=-P\left\{\sum^{n}_{i,j=1}\partial_{x_{j}}(a_{ij}\partial_{x_{i}}u)\right\},
\end{equation*}
$B$ is a nonlinear operator from $D(B):=H^{2}(\Omega)\cap V$ to $H$.
$A$ and $B$ are assumed to the following properties:
\begin{itemize}
\item[(A.1)]$a_{ij} \in C^{0,1}(\overline{\Omega})$ $(i,j=1,\cdots, n)$.
\item[(A.2)]$a_{ij}=a_{ji} \ \mathrm{on} \ \overline{\Omega}$ $(i,j=1,\cdots, n)$.
\item[(A.3)]There exists a positive constant $a$ such that
\begin{equation*}
\sum^{n}_{i,j=1}a_{ij}(x)\xi_{i}\xi_{j}\geq a|\xi|^{2}
\end{equation*}
for any $x \in \overline{\Omega}$, $\xi \in \mathbb{R}^{n}$.
\item[(A.4)]Let $(u,v)_{D(A)}=(Au,Av)_{L^{2}(\Omega)}$, $\|u\|_{D(A)}=((u,u)_{D(A)})^{1/2}$.
Then $\|\cdot\|_{D(A)}$ is equivalent to $\|\cdot\|_{H^{2}(\Omega)}$ as a norm in $D(A)$: there exist positive constants $a_{1}$ and $a_{2}$ such that
\begin{equation*}
a_{1}\|u\|_{H^{2}(\Omega)} \leq \|u\|_{D(A)} \leq a_{2}\|u\|_{H^{2}(\Omega)}
\end{equation*}
for any $u \in D(A)$.
\item[(B.1)]$B0=0$.
\item[(B.2)]There exist constants $C_{B}>0$ and $p>1$ such that
\begin{equation*}
\|Bu-Bv\|_{L^{2}(\Omega)}\leq C_{B}(1+\|u\|_{H^{2}(\Omega)}^{p-1}+\|v\|_{H^{2}(\Omega)}^{p-1})\|u-v\|_{H^{1}(\Omega)}
\end{equation*}
for any $u, v \in D(B)$.
\end{itemize}
Let us introduce the following scalar product and norm in $V$:
\begin{equation*}
(u,v)_{a}=\sum^{n}_{i,j=1}(a_{ij}\partial_{x_{i}}u,\partial_{x_{j}}v)_{L^{2}(\Omega)}, \ \|u\|_{a}=((u,u)_{a})^{1/2}.
\end{equation*}
It follows from (A.3) and the Schwarz inequality that $\|\cdot\|_{a}$ is equivalent to $\|\cdot\|_{H^{1}(\Omega)}$ as a norm in $V$: there exist positive constants $a_{3}$ and $a_{4}$ such that
\begin{equation*}
a_{3}\|u\|_{H^{1}(\Omega)} \leq \|u\|_{a} \leq a_{4}\|u\|_{H^{1}(\Omega)}
\end{equation*}
for any $u \in V$.
Strong solutions of (1.1), (1.2) are defined as follows:
\begin{definition}
Let $f \in L^{2}((0,\infty);H)$, $u_{0} \in V$.
Then $u$ is called a strong solution of (1.1) if $u \in L^{2}((0,\infty);D(A))\cap C_{b}([0,\infty);V)$, $d_{t}u \in L^{2}((0,\infty);H)$, $u$ satisfies (1.1).
Let $\mathcal{S}(f, u_{0})$ be the set of all functions which are strong solutions of (1.1) with $f$ and $u_{0}$.
\end{definition}
\begin{definition}
Let $\bar{f} \in H$.
Then $\bar{u}$ is called a strong solution of (1.2) if $\bar{u} \in D(A)$, $\bar{u}$ satisfies (1.2).
Let $\mathcal{S}(\bar{f})$ be the set of all functions which are solutions of (1.2) with $\bar{f}$.
\end{definition}

\subsection{Main results}
For any $N \in \mathbb{Z}$, $N\geq 1$, $x \in \overline{\Omega}$ and $u \in D(A)$, $E_{N}$, $d_{N}(x)$, $d_{N}$ and $\eta_{N}(u)$ are defied as follows:
\begin{equation*}
E_{N}=\{x_{1}, \cdots, x_{N} \ ; \ x_{j} \in \overline{\Omega} \ (j=1, \cdots, N)\},
\end{equation*}
\begin{equation*}
d_{N}(x)=\min_{j=1, \cdots, N}|x-x_{j}|,
\end{equation*}
\begin{equation*}
d_{N}=\max_{x \in \overline{\Omega}}d_{N}(x),
\end{equation*}
\begin{equation*}
\eta_{N}(u)=\max_{j=1, \cdots, N}|u(x_{j})|.
\end{equation*}
We can consider $E_{N}$ and $d_{N}$ as the set of determining nodes and the density of $E_{N}$ in $\Omega$ respectively.
It is essential for our main results to be assumed that
\begin{itemize}
\item[(H.1)]$\mathcal{S}(\bar{f})\neq\emptyset$ for any $\bar{f} \in H$.
\item[(H.2)]There exists a positive constant $M(\bar{f})$ for any $\bar{f} \in H$ such that $\|\bar{u}\|_{D(A)}\leq M(\bar{f})$ for any $\bar{u} \in \mathcal{S}(\bar{f})$.
\item[(H.3)]$\mathcal{S}(f,u_{0})\neq\emptyset$ for any $f \in L^{\infty}((0,\infty);H)$, $u_{0} \in V$.
\item[(H.4)]There exists a positive constant $M(f,u_{0},t_{0})$ for any $f \in L^{\infty}((0,\infty);H)$, $u_{0} \in V$, $t_{0}>0$ such that $\|u\|_{C_{b}([t_{0},\infty);D(A))}\leq M(f,u_{0},t_{0})$ for any $u \in \mathcal{S}(f,u_{0})$.
\end{itemize}
Our main results are given by Theorems 2.1--2.3.
\begin{theorem}
Let $n=2, 3$, $\bar{f} \in H$, $\bar{u},\bar{v} \in \mathcal{S}(\bar{f})$, and assume $\mathrm{(H.1)}$, $\mathrm{(H.2)}$.
Then there exists a positive constant $\delta_{1}$ depending only on $\Omega$, $A$, $B$ and $M(\bar{f})$ such that if $0<d_{N}\leq\delta_{1}$, $\bar{u}(x_{j})=\bar{v}(x_{j})$ $(j=1, \cdots, N)$, then $\bar{u}=\bar{v}$ in $\Omega$.
\end{theorem}
\begin{theorem}
Let $n=2, 3$, $f \in L^{\infty}((0,\infty);H)$, $u_{0} \in V$, $u \in \mathcal{S}(f,u_{0})$, and assume $\mathrm{(H.2)}$--$\mathrm{(H.4)}$, $f(t)\rightarrow  f_{\infty} \in H$ in $H$ as $t \to \infty$.
Then there exists a positive constant $\delta_{2}$ depending only on $\Omega$, $A$, $B$, $f$, $M(f_{\infty})$ and $M(f,u_{0},t_{0})$ such that if $0<d_{N}\leq\delta_{2}$, $u(x_{j},t)\rightarrow \xi_{j} \in \mathbb{R}$ as $t \to \infty$ $(j=1, \cdots, N)$, then $\mathrm{(1.2)}$ has uniquely a strong solution $u_{\infty} \in \mathcal{S}(f_{\infty})$ satisfying $u(t)\rightarrow  u_{\infty}$ in $V\cap C^{0,\mu}(\overline{\Omega})$ as $t\rightarrow \infty$ for any $0<\mu<1/2$ and $u_{\infty}(x_{j})=\xi_{j}$ $(j=1, \cdots, N)$.
\end{theorem}
\begin{theorem}
Let $n=2, 3$, $f, g \in L^{\infty}((0,\infty);H)$, $u_{0}, v_{0} \in V$, $u \in \mathcal{S}(f,u_{0})$, $v \in \mathcal{S}(g,v_{0})$, and assume $\mathrm{(H.3)}$, $\mathrm{(H.4)}$.
Then there exists a positive constant $\delta_{3}$ depending only on $\Omega$, $A$, $B$, $f$, $g$, $M(f,u_{0},t_{0})$ and $M(g,v_{0},t_{0})$ such that if $0<d_{N}\leq\delta_{3}$, $u(x_{j},t)-v(x_{j},t)\rightarrow 0$ $(j=1, \cdots, N)$, $f(t)-g(t)\rightarrow 0$ in $H$ as $t\rightarrow \infty$, then $u(t)-v(t)\rightarrow 0$ in $V\cap C^{0,\mu}(\overline{\Omega})$ as $t\rightarrow \infty$ for any $0<\mu<1/2$.
\end{theorem}

\subsection{Lemmas}
We will state lemmas for our main results.
It is important for our main results that the following inequalities relate $\|\cdot\|_{C(\overline{\Omega})}$, $\|\cdot\|_{L^{2}(\Omega)}$ and $\|\cdot\|_{H^{1}(\Omega)}$ to $d_{N}$.
\begin{lemma}
Let $n=2,3$.
Then there exists a positive constant $C_{1}$ depending only on $\Omega$ such that
\begin{equation}
\|u\|_{C(\overline{\Omega})}\leq \eta_N(u)+C_{1}d^{1/2}_{N}\|u\|_{D(A)}
\end{equation}
for any $u \in D(A)$.
\end{lemma}
\begin{proof}
It is \cite[Lemma 2.1]{Foias 2}.
\end{proof}
\begin{lemma}
Let $n=2,3$.
Then there exist positive constants $C_{2}$ and $C_{3}$ depending only on $\Omega$ such that
\begin{equation}
\|u\|_{L^{2}(\Omega)}\leq C_{2}\eta_{N}(u)+C_{3}d^{1/2}_{N}\|u\|_{D(A)}
\end{equation}
for any $u \in D(A)$.
\end{lemma}
\begin{proof}
It is \cite[Lemma 2.1]{Foias 2}.
\end{proof}
\begin{lemma}
Let $n=2,3$.
Then there exist positive constants $C_{4}$ and $C_{5}$ depending only on $\Omega$ such that
\begin{equation}
\|u\|_{H^{1}(\Omega)}\leq C_{4}d^{-1/4}_{N}\eta_{N}(u)+C_{5}d^{1/4}_{N}\|u\|_{D(A)}
\end{equation}
for any $u \in D(A)$.
\end{lemma}
\begin{proof}
It is \cite[Lemma 2.1]{Foias 2}.
\end{proof}

\section{Proof of Theorems 2.1--2.3}
We will prove our main results which appeared in subsection 2.3.

\subsection{Proof of Theorem \rm{2.1}}
$\bar{v} \in \mathcal{S}(\bar{f})$ satisfies the following equation:
\begin{equation}
A\bar{v}+B\bar{v}=\bar{f}.
\end{equation}
We subtract (3.1) from (1.2), and obtain that
\begin{equation*}
A(\bar{u}-\bar{v})+B\bar{u}-B\bar{v}=0.
\end{equation*}
By taking the $H$-norm of this equality and (B.2), we have that
\begin{equation*}
\begin{split}
\|A(\bar{u}-\bar{v})\|_{L^{2}(\Omega)}&=\|B\bar{u}-B\bar{v}\|_{L^{2}(\Omega)} \\
&\leq C_{B}(\|\bar{u}\|^{p-1}_{H^{2}(\Omega)}+\|\bar{v}\|^{p-1}_{H^{2}(\Omega)})\|\bar{u}-\bar{v}\|_{H^{1}(\Omega)},
\end{split}
\end{equation*}
\begin{equation*}
\|\bar{u}-\bar{v}\|_{D(A)}\leq 2C_{B}M(\bar{f})^{p-1}\|\bar{u}-\bar{v}\|_{H^{1}(\Omega)}.
\end{equation*}
It follows from $\bar{u}(x_{j})-\bar{v}(x_{j})=0$ $(j=1, \cdots, N)$ that $\eta(\bar{u}-\bar{v})=0$.
This equality and (2.3) imply that
\begin{equation*}
\|\bar{u}-\bar{v}\|_{H^{1}(\Omega)}\leq C_{5}d^{1/4}_{N}\|\bar{u}-\bar{v}\|_{D(A)}.
\end{equation*}
Therefore, we obtain that
\begin{equation*}
\|\bar{u}-\bar{v}\|_{D(A)}\leq 2C_{B}C_{5}M(\bar{f})^{p-1}d^{1/4}_{N}\|\bar{u}-\bar{v}\|_{D(A)},
\end{equation*}
\begin{equation*}
(1-2C_{B}C_{5}M(\bar{f})^{p-1}d^{1/4}_{N})\|\bar{u}-\bar{v}\|_{D(A)} \leq 0.
\end{equation*}
If it is known that
\begin{equation*}
1-2C_{B}C_{5}M(\bar{f})^{p-1}d^{1/4}_{N}>0,
\end{equation*}
\begin{equation}
0<d_{N}<\frac{1}{(2C_{B}C_{5}M(\bar{f})^{p-1})^{4}},
\end{equation}
then we can conclude that $\bar{u}=\bar{v}$ in $\Omega$.
The sufficient condition for (3.2) is
\begin{equation*}
0<\delta_{1}<\frac{1}{(2C_{B}C_{5}M(\bar{f})^{p-1})^{4}}.
\end{equation*}

\subsection{Proof of Theorem \rm{2.2}}
First, we obtain an energy-type inequality.
We consider two times $t$, $s$ satisfying $t<s$, and write $s=t+\tau$ $(\tau >0)$.
Let $v(t)=u(t+\tau)$, $g(t)=f(t+\tau)$.
Then $v$ satisfies the following equation:
\begin{equation}
d_{t}v+Av+Bv=g.
\end{equation}
We subtract (3.3) from the first equation of (1.1), and have that
\begin{equation}
d_{t}(u-v)+A(u-v)+Bu-Bv=f-g.
\end{equation}
We take the $H$-scalar product of (3.4) with $A(u-v)$, and obtain from (B.2) that
\begin{equation*}
(d_{t}(u-v), A(u-v))_{L^{2}(\Omega)}+\|A(u-v)\|^{2}_{L^{2}(\Omega)}+(Bu-Bv, A(u-v))_{L^{2}(\Omega)}=(f-g, A(u-v))_{L^{2}(\Omega)},
\end{equation*}
\begin{equation*}
\begin{split}
\frac{1}{2}d_{t}(\|u-v\|^{2}_{a})+\|u-v\|^{2}_{D(A)}=&-(Bu-Bv, A(u-v))_{L^{2}(\Omega)}+(f-g, A(u-v))_{L^{2}(\Omega)} \\
\leq& |(Bu-Bv,A(u-v))_{L^{2}(\Omega)}|+|(f-g,A(u-v))_{L^{2}(\Omega)}| \\
\leq& \|Bu-Bv\|_{L^{2}(\Omega)}\|A(u-v)\|_{L^{2}(\Omega)}+\|f-g\|_{L^{2}(\Omega)}\|A(u-v)\|_{L^{2}(\Omega)} \\
\leq& C_{B}(\|u\|^{p-1}_{H^{2}(\Omega)}+\|v\|^{p-1}_{H^{2}(\Omega)})\|u-v\|_{H^{1}(\Omega)}\|u-v\|_{D(A)} \\
&+\frac{1}{4}\|u-v\|^{2}_{D(A)}+\|f-g\|^{2}_{L^{2}(\Omega)} \\
\leq& 2C_{B}M(f,u_{0},t_{0})^{p-1}\|u-v\|_{H^{1}(\Omega)}\|u-v\|_{D(A)} \\
&+\frac{1}{4}\|u-v\|^{2}_{D(A)}+\|f-g\|^{2}_{L^{2}(\Omega)}.
\end{split}
\end{equation*}
Let us notice from (2.3) that
\begin{equation*}
\|u-v\|_{H^{1}(\Omega)}\leq C_{4}d^{-1/4}_{N}\eta_{N}(u-v)+C_{5}d^{1/4}_{N}\|u-v\|_{D(A)}.
\end{equation*}
Then we obtain from the above two inequalities that
\begin{equation*}
\begin{split}
\frac{1}{2}d_{t}(\|u-v\|^{2}_{a})+\|u-v\|^{2}_{D(A)}\leq& 2C_{B}C_{4}M(f,u_{0},t_{0})^{p-1}d^{-1/4}_{N}\eta_{N}(u-v)\|u-v\|_{D(A)} \\
&+2C_{B}C_{5}M(f,u_{0},t_{0})^{p-1}d^{1/4}_{N}\|u-v\|^{2}_{D(A)} \\
&+\frac{1}{4}\|u-v\|^{2}_{D(A)}+\|f-g\|^{2}_{L^{2}(\Omega)} \\
\leq& 4C^{2}_{B}C^{2}_{4}M(f,u_{0},t_{0})^{2(p-1)}d^{-1/2}_{N}\eta_{N}(u-v)^{2} \\
&+2C_{B}C_{5}M(f,u_{0},t_{0})^{p-1}d^{1/4}_{N}\|u-v\|^{2}_{D(A)} \\
&+\frac{1}{2}\|u-v\|^{2}_{D(A)}+\|f-g\|^{2}_{L^{2}(\Omega)},
\end{split}
\end{equation*}
\begin{equation}
\begin{split}
d_{t}(\|u-v\|^{2}_{a})&+(1-4C_{B}C_{5}M(f,u_{0},t_{0})^{p-1}d^{1/4}_{N})\|u-v\|^{2}_{D(A)} \\
&\leq 8C^{2}_{B}C^{2}_{4}M(f,u_{0},t_{0})^{2(p-1)}d^{-1/2}_{N}\eta_{N}(u-v)^{2}+2\|f-g\|^{2}_{L^{2}(\Omega)}.
\end{split}
\end{equation}
We assume that
\begin{equation*}
1-4C_{B}C_{5}M(f,u_{0},t_{0})^{p-1}d^{1/4}_{N}>0,
\end{equation*}
\begin{equation}
0<d_{N}<\frac{1}{(4C_{B}C_{5}M(f,u_{0},t_{0})^{p-1})^{4}},
\end{equation}
and set
\begin{equation*}
\lambda=\frac{a^{2}_{1}}{a^{2}_{4}}(1-4C_{B}C_{5}M(f,u_{0},t_{0})^{p-1}d^{1/4}_{N}),
\end{equation*}
\begin{equation*}
h(t)=8C^{2}_{B}C^{2}_{4}M(f,u_{0},t_{0})^{2(p-1)}d^{-1/2}_{N}\eta_{N}(u-v)^{2}+2\|f-g\|^{2}_{L^{2}(\Omega)}.
\end{equation*}
Since $\lambda>0$ from the definition of $\lambda$, (3.5) becomes
\begin{equation}
d_{t}(\|(u-v)(t)\|^{2}_{a})+\lambda\|(u-v)(t)\|^{2}_{a}\leq h(t)
\end{equation}
for any $t\geq t_{0}$.

We show by an energy-type inequality that $\{u(t)\}_{t \geq t_{0}}$ is a Cauchy sequence in $V$.
It follows from $f(t)\rightarrow f_{\infty}$ in $H$ as $t\rightarrow \infty$ and $u(x_{j},t)\rightarrow \xi_{j}$ as $t\rightarrow \infty$ $(j=1, \cdots, N)$ that $h\rightarrow 0$ as $t\rightarrow \infty$.
Hence, there exists a positive constant $t(\varepsilon)$ for any positive constant $\varepsilon$ such that $|h(t)|\leq\varepsilon$ for any $t\geq t(\varepsilon)$.
It is derived from (3.7) that we have the following inequality:
\begin{equation}
d_{t}(\|(u-v)(t)\|^{2}_{a})+\lambda\|(u-v)(t)\|^{2}_{a}\leq\varepsilon
\end{equation}
for any $t\geq t(\varepsilon)$.
The Gronwall lemma and (3.8) imply that
\begin{equation*}
\|(u-v)(t)\|^{2}_{a}\leq \|(u-v)(t(\varepsilon))\|^{2}_{a}e^{-\lambda(t-t(\varepsilon))}+\frac{\varepsilon}{\lambda}(1-e^{-\lambda(t-t(\varepsilon))}),
\end{equation*}
\begin{equation}
\|u(t)-u(s)\|^{2}_{a}\leq \|(u-v)(t(\varepsilon))\|^{2}_{a}e^{-\lambda(t-t(\varepsilon))}+\frac{\varepsilon}{\lambda}(1-e^{-\lambda(t-t(\varepsilon))})
\end{equation}
for any $t\geq t(\varepsilon)$.
We take $t$, $s$ to infinity in (3.9), and obtain that
\begin{equation*}
\limsup_{t, s\rightarrow \infty}\|u(t)-u(s)\|^{2}_{a}\leq\frac{\varepsilon}{\lambda}.
\end{equation*}
Since $\varepsilon$ is an arbitrary positive constant, we conclude that $u(t)-v(t)\rightarrow 0$ in $V$ as $t, s\rightarrow \infty$, that is, $\{u(t)\}_{t\geq t_{0}}$ is a Cauchy sequence in $V$.
The completeness of $V$ implies that there exists $u_{\infty} \in V$ satisfying
\begin{equation}
u(t)\rightarrow u_{\infty} \ \mathrm{in} \ V \ \mathrm{as} \ t\rightarrow \infty.
\end{equation}

Finally, we prove that $u_{\infty} \in \mathcal{S}(f_{\infty})$ and $u_{\infty}(x_{j})=\xi_{j}$ $(j=1, \cdots, N)$.
$\{u(t)\}_{t\geq t_{0}}$ is bounded in $D(A)$ because of (H.4).
$D(A)$ is compactly embedded in $C^{0,\mu}(\overline{\Omega})$ for any $0<\mu<1/2$ from the Rellich-Kondrachov theorem.
Hence, we conclude from (3.10) that
\begin{equation}
u(t)\rightarrow u_{\infty} \ \mathrm{in} \ C^{0,\mu}(\overline{\Omega}) \ as \ t\rightarrow \infty.
\end{equation}
$u(x_{j},t)\rightarrow \xi_{j}$ as $t\rightarrow \infty$ $(j=1,\cdots, N)$ and (3.11) imply that $u_{\infty}(x_{j})=\xi_{j}$ $(j=1, \cdots, N)$.
By taking $t$ to infinity in $(1.1)_{1}$, the straightforward argument shows that $u_{\infty} \in \mathcal{S}(f_{\infty})$.
Let us choose $\delta_{2}\leq\delta_{1}(M(f_{\infty}))$.
Then (1.2) has uniquely a strong solution $u_{\infty} \in \mathcal{S}(f_{\infty})$ satisfying $u_{\infty}(x_{j})=\xi_{j}$ $(j=1, \cdots, N)$ from Theorem 2.1.
Therefore, the sufficient condition for (3.6) and desired properties of $u_{\infty}$ is
\begin{equation*}
0<\delta_{2}<\min\left\{\delta_{1}(M(f_{\infty})), \frac{1}{(4C_{B}C_{5}M(f,u_{0},t_{0})^{p-1})^{4}}\right\}.
\end{equation*}

\subsection{Proof of Theorem \rm{2.3}}
Similarly to the proof of Theorem 2.2, we obtain an energy-type inequality.
$v$ satisfies the following equation:
\begin{equation}
d_{t}v+Av+Bv=g.
\end{equation}
We subtract (3.12) from the first equation of (1.1), and have that
\begin{equation}
d_{t}(u-v)+A(u-v)+Bu-Bv=f-g.
\end{equation}
By taking the $H$-scalar product of (3.13) with $A(u-v)$, we obtain from (B.2) that
\begin{equation*}
(d_{t}(u-v), A(u-v))_{L^{2}(\Omega)}+\|A(u-v)\|^{2}_{L^{2}(\Omega)}+(Bu-Bv, A(u-v))_{L^{2}(\Omega)}=(f-g, A(u-v))_{L^{2}(\Omega)},
\end{equation*}
\begin{equation*}
\begin{split}
\frac{1}{2}d_{t}(\|u-v\|^{2}_{a})+\|u-v\|^{2}_{D(A)}=&-(Bu-Bv, A(u-v))_{L^{2}(\Omega)}+(f-g, A(u-v))_{L^{2}(\Omega)} \\
\leq& |(Bu-Bv,A(u-v))_{L^{2}(\Omega)}|+|(f-g, A(u-v))_{L^{2}(\Omega)}| \\
\leq& \|Bu-Bv\|_{L^{2}(\Omega)}\|A(u-v)\|_{L^{2}(\Omega)}+\|f-g\|_{L^{2}(\Omega)}\|A(u-v)\|_{L^{2}(\Omega)} \\
\leq& C_{B}(\|u\|^{p-1}_{H^{2}(\Omega)}+\|v\|^{p-1}_{H^{2}(\Omega)})\|u-v\|_{H^{1}(\Omega)}\|u-v\|_{D(A)} \\
&+\frac{1}{4}\|u-v\|^{2}_{D(A)}+\|f-g\|^{2}_{L^{2}(\Omega)} \\
\leq& C_{B}(M(f,u_{0},t_{0})^{p-1}+M(g,v_{0},t_{0})^{p-1})\|u-v\|_{H^{1}(\Omega)}\|u-v\|_{D(A)} \\
&+\frac{1}{4}\|u-v\|^{2}_{D(A)}+\|f-g\|^{2}_{L^{2}(\Omega)}.
\end{split}
\end{equation*}
It follows from (2.3) that
\begin{equation*}
\|u-v\|_{H^{1}(\Omega)}\leq C_{4}d^{-1/4}_{N}\eta_{N}(u-v)+C_{5}d^{1/4}_{N}\|u-v\|_{D(A)}.
\end{equation*}
Let $M(p,t_{0})=M(f,u_{0},t_{0})^{p-1}+M(g,v_{0},t_{0})^{p-1}$.
Then it is derived from the above two inequalities that
\begin{equation*}
\begin{split}
\frac{1}{2}d_{t}(\|u-v\|^{2}_{a})+\|u-v\|^{2}_{D(A)}\leq& C_{B}C_{4}M(p,t_{0})d^{-1/4}_{N}\eta_{N}(u-v)\|u-v\|_{D(A)} \\
&+C_{B}C_{5}M(p,t_{0})d^{1/4}_{N}\|u-v\|^{2}_{D(A)} \\
&+\frac{1}{4}\|u-v\|^{2}_{D(A)}+\|f-g\|^{2}_{L^{2}(\Omega)} \\
\leq& C^{2}_{B}C^{2}_{4}M(p,t_{0})^{2}d^{-1/2}_{N}\eta_{N}(u-v)^{2} \\
&+C_{B}C_{5}M(p,t_{0})d^{1/4}_{N}\|u-v\|^{2}_{D(A)} \\
&+\frac{1}{2}\|u-v\|^{2}_{D(A)}+\|f-g\|^{2}_{L^{2}(\Omega)},
\end{split}
\end{equation*}
\begin{equation}
\begin{split}
d_{t}(\|u-v\|^{2}_{a})&+(1-2C_{B}C_{5}M(p,t_{0})d^{1/4}_{N})\|u-v\|^{2}_{D(A)} \\
&\leq 2C^{2}_{B}C^{2}_{4}M(p,t_{0})^{2}d^{-1/2}_{N}\eta_{N}(u-v)^{2}+2\|f-g\|^{2}_{L^{2}(\Omega)}.
\end{split}
\end{equation}
We assume that
\begin{equation*}
1-2C_{B}C_{5}M(p,t_{0})d^{1/4}_{N}>0,
\end{equation*}
\begin{equation}
0<d_{N}<\frac{1}{(2C_{B}C_{5}M(p,t_{0}))^{4}},
\end{equation}
and set
\begin{equation*}
\lambda=\frac{a^{2}_{1}}{a^{2}_{4}}(1-2C_{B}C_{5}M(p,t_{0})d^{1/4}_{N}),
\end{equation*}
\begin{equation*}
h(t)=2C^{2}_{B}C^{2}_{4}M(p,t_{0})^{2}d^{-1/2}_{N}\eta_{N}((u-v)(t))^{2}+\|(f-g)(t)\|^{2}_{L^{2}(\Omega)}.
\end{equation*}
Then $\lambda>0$, we obtain from (3.14) that
\begin{equation}
d_{t}(\|(u-v)(t)\|^{2}_{a})+\lambda\|(u-v)(t)\|^{2}_{a}\leq h(t)
\end{equation}
for any $t\geq t_{0}$.

We show by an energy-type inequality that $u(t)-v(t)\rightarrow 0$ in $V$ as $t\rightarrow \infty$.
It follows from $f(t)-g(t)\rightarrow 0$ in $H$ as $t\rightarrow \infty$ and $u(x_{j},t)-v(x_{j},t)\rightarrow 0$ as $t\rightarrow \infty$ $(j=1,\cdots, N)$ that $h\rightarrow 0$ as $t\rightarrow \infty$.
Hence, there exists a positive constant $t(\varepsilon)$ for any positive constant $\varepsilon$ such that $|h(t)|\leq\varepsilon$ for any $t\geq t(\varepsilon)$.
It is derived from (3.16) that we have the following inequality:
\begin{equation}
d_{t}(\|(u-v)(t)\|^{2}_{a})+\lambda\|(u-v)(t)\|^{2}_{a}\leq\varepsilon
\end{equation}
for any $t \geq t(\varepsilon)$.
The Gronwall lemma and (3.17) imply that
\begin{equation}
\|(u-v)(t)\|^{2}_{a}\leq\|(u-v)(t(\varepsilon))\|^{2}_{a}e^{-\lambda(t-t(\varepsilon))}+\frac{\varepsilon}{\lambda}(1-e^{-\lambda(t-t(\varepsilon))})
\end{equation}
for any $t \geq t(\varepsilon)$.
We take $t$ to infinity in (3.18), and obtain that
\begin{equation*}
\limsup_{t\rightarrow \infty}\|(u-v)(t)\|^{2}_{a}\leq\frac{\varepsilon}{\lambda}.
\end{equation*}
Since $\varepsilon$ is an arbitrary positive constant, we conclude that
\begin{equation}
u(t)-v(t)\rightarrow 0 \ \mathrm{in} \ V \ as \ t\rightarrow \infty.
\end{equation}

Finally, we prove that $u(t)-v(t)\rightarrow 0$ in $C^{0,\mu}(\overline{\Omega})$ as $t\rightarrow \infty$ for any $0<\mu<1/2$.
$\{(u-v)(t)\}_{t\geq t_{0}}$ is bounded in $D(A)$ because of (H.4).
$D(A)$ is compactly embedded in $C^{0,\mu}(\overline{\Omega})$ from the Rellich-Kondrachov theorem.
Hence, we conclude from (3.19) that
\begin{equation*}
u(t)-v(t)\rightarrow 0 \ \mathrm{in} \ C^{0,\mu}(\overline{\Omega}) \ as \ t\rightarrow \infty.
\end{equation*}
The sufficient condition for (3.15) is
\begin{equation*}
0<\delta_{3}<\frac{1}{(2C_{B}C_{5}M(p,t_{0}))^{4}}.
\end{equation*}

\section{Applications}
We will apply our main results to the semilinear heat equation and the Navier-Stokes equations in subsections 4.2 and 4.3 respectively after some preliminaries in subsection 4.1.

\subsection{Sectorial operators in $L^{2}$ and analytic semigroups on $L^{2}$}
The theory of analytic semigroups on $L^{2}(\Omega)$ and fractional powers of sectorial operators are introduced as follows: Let $(X,\|\cdot\|_{X})$ and $(Y,\|\cdot\|_{Y})$ be Banach spaces.
$\mathcal{B}(X;Y)$ is the Banach space of all linear operators from $X$ into $Y$ which are bounded in $X$, $\mathcal{B}(X)=\mathcal{B}(X;X)$.
The norm in $\mathcal{B}(X;Y)$ is denoted by $\|\cdot\|_{\mathcal{B}(X;Y)}$, that is,
\begin{equation*}
\|A\|_{\mathcal{B}(X;Y)}=\sup_{x \in X\setminus\{0\}}\frac{\|Ax\|_{Y}}{\|x\|_{X}}.
\end{equation*}
Let $A$ be a sectorial operator in $L^{2}(\Omega)$ defined as in \cite[Definition 1.3.1]{Henry}, $D(A)\subset H^{2}(\Omega)$, $\mathrm{Re}\sigma(A)>0$, where $\mathrm{Re}\sigma(A)>0$ means that $\mathrm{Re}\lambda>0$ for any $\lambda \in \sigma(A)$.
It is well known in \cite[Theorem 1.3.4 and Definition 1.4.1]{Henry}, \cite[Theorem 2.5.2 and Definition 2.6.7]{Pazy} that $-A$ generates an uniformly bounded analytic semigroup $\{e^{-tA}\}_{t\geq 0}$ on $L^{2}(\Omega)$, fractional powers $A^{\alpha}$ of $A$ can be defined for any $\alpha\geq 0$, $A^{0}=I$.
Let us introduce the Hilbert space $D(A^{\alpha})$ with the scalar product $(u,v)_{D(A^{\alpha})}=(A^{\alpha}u,A^{\alpha}v)_{L^{2}(\Omega)}$ and the norm $\|u\|_{D(A^{\alpha})}=((u,u)_{D(A^{\alpha})})^{1/2}$ for any $0\leq\alpha\leq 1$.

We state some lemmas concerning sectorial operators in $L^{2}(\Omega)$.
See, for example, \cite[Chapter 1]{Henry}, \cite[Chapter 2]{Pazy} on the theory of analytic semigroups on Banach spaces and fractional powers of sectorial operators.
\begin{lemma}
Let $\alpha\geq 0$, $0<\lambda<\lambda_{1}$, $\lambda_{1}=\min\{\lambda>0 \ ; \ \lambda \in \mathrm{Re}\sigma(A)\}$.
Then there exists a positive constant $C_{\alpha,\lambda}$ depending only on $n$, $\Omega$, $A$, $\alpha$ and $\lambda$ such that
\begin{equation*}
\|A^{\alpha}e^{-tA}u\|_{\mathcal{B}(L^{2}(\Omega))}\leq C_{\alpha,\lambda}t^{-\alpha}e^{-\lambda t}.
\end{equation*}
\end{lemma}
\begin{proof}
It is \cite[Theorem 1.4.3]{Henry}.
\end{proof}
\begin{lemma}
Let $0\leq\alpha\leq 1$.
Then $D(A^{\alpha})$ is continuously embedded in $L^{q}(\Omega)$ if $1/2-2\alpha/n\leq 1/q\leq 1/2$.
\end{lemma}
\begin{proof}
It is \cite[Theorem 1.6.1]{Henry}.
\end{proof}

\subsection{Semilinear heat equation}
The initial-boundary value problem for the semilinear heat equation is described as follows:
\begin{equation}
\begin{split}
&\partial_{t}u-k\Delta u-|u|^{p-1}u=f & \mathrm{in} \ \Omega\times(0,\infty), \\
&u|_{t=0}=u_{0} & \mathrm{in} \ \Omega, \\
&u|_{\partial\Omega}=0 & \mathrm{on} \ \partial\Omega\times(0,\infty),
\end{split}
\end{equation}
where $k>0$, $p>1$, $f$ is an external force, $u_{0}$ is an initial data of $u$.

Let $H=L^{2}(\Omega)$, $V=H^{1}_{0}(\Omega)$, $P=I_{2}$, where $I_{2}$ is the identity operator in $L^{2}(\Omega)$.
Then we can utilize the strong formulation to rewrite (4.1) by
\begin{equation}
\begin{split}
&d_{t}u+Au+b(u)=f & \mathrm{in} \ L^{2}((0,\infty);L^{2}(\Omega)), \\
&u(0)=u_{0} & \mathrm{in} \ H^{1}_{0}(\Omega),
\end{split}
\end{equation}
where $Au=-k\Delta u$, $b(u)=-|u|^{p-1}u$.
It follows from \cite[Theorem 8.12]{Gilbarg} that $A$ satisfies (A.1)--(A.4).
Moreover, $b$ satisfies the following properties:
\begin{itemize}
\item[(b.1)]$b(0)=0$.
\item[(b.2)]There exists a positive constant $C_{b}$ depending only on $p$ such that
\begin{equation*}
|b(u)-b(v)|\leq C_{b}(|u|^{p-1}+|v|^{p-1})|u-v|
\end{equation*}
for any $u,v \in \mathbb{R}$.
\end{itemize}
It is assured by the following lemma that $Bu=b(u)$ satisfies $\mathrm{(B.1)}$, $\mathrm{(B.2)}$.
\begin{lemma}
Let $n=2, 3$, $1<p\leq n/(n-2)$.
Then there exists a positive constant $C_{B}$ depending only on $\Omega$ and $p$ such that
\begin{equation}
\|b(u)-b(v)\|_{L^{2}(\Omega)}\leq C_{B}(\|u\|^{p-1}_{H^{1}(\Omega)}+\|v\|^{p-1}_{H^{1}(\Omega)})\|u-v\|_{H^{1}(\Omega)}
\end{equation}
for any $u,v \in H^{1}(\Omega)$.
\end{lemma}
\begin{proof}
By taking the $L^{2}$-norm of (b.2), it follows from the H\"{o}lder inequality and the Minkowski inequality that
\begin{equation*}
\begin{split}
\|b(u)-b(v)\|_{L^{2}(\Omega)}&\leq C_{b}\|(|u|^{p-1}+|v|^{p-1})|u-v|\|_{L^{2}(\Omega)} \\
&\leq C_{b}\||u|^{p-1}+|v|^{p-1}\|_{L^{2p/(p-1)}(\Omega)}\|u-v\|_{L^{2p}(\Omega)} \\
&\leq C_{b}(\||u|^{p-1}\|_{L^{2p/(p-1)}(\Omega)}+\||v|^{p-1}\|_{L^{2p/(p-1)}(\Omega)})\|u-v\|_{L^{2p}(\Omega)} \\
&=C_{b}(\|u\|^{p-1}_{L^{2p}(\Omega)}+\|v\|^{p-1}_{L^{2p}(\Omega)})\|u-v\|_{L^{2p}(\Omega)},
\end{split}
\end{equation*}
\begin{equation}
\|b(u)-b(v)\|_{L^{2}(\Omega)}\leq C_{b}(\|u\|^{p-1}_{L^{2p}(\Omega)}+\|v\|^{p-1}_{L^{2p}(\Omega)})\|u-v\|_{L^{2p}(\Omega)}
\end{equation}
for any $u,v \in H^{1}(\Omega)$.
$H^{1}(\Omega)$ is continuously embedded in $L^{2p}(\Omega)$ from the Sobolev embedding theorem.
Hence, it is clear that (4.3) is established by (4.4).
\end{proof}
It can be easily seen from the following theorems that $\mathrm{(H.3)}$, $\mathrm{(H.4)}$ hold for (4.2) under appropriate assumptions for $p$, $f$ and $u_{0}$.
\begin{theorem}
Let $n=2, 3$, $1<p\leq n/(n-2)$, $f \in L^{2}((0,\infty);L^{2}(\Omega))$, $u_{0} \in H^{1}_{0}(\Omega)$.
Then there exists a (small) positive constant $\varepsilon_{1}$ and $\varepsilon_{2}$ depending only on $\Omega$, $k$ and $p$ such that $(4.2)$ has uniquely a strong solution provided that $\|f\|_{L^{\infty}((0,\infty);L^{2}(\Omega))}\leq\varepsilon_{1}$, $\|u_{0}\|_{H^{1}(\Omega)}\leq\varepsilon_{2}$.
\end{theorem}
\begin{proof}
It is well known in \cite[Theorem 3.2.1]{Ladyzhenskaya} that
\begin{equation}
\begin{split}
&d_{t}u+Au=g & \mathrm{in} \ L^{2}((0,T);L^{2}(\Omega)), \\
&u(0)=u_{0} & \mathrm{in} \ H^{1}_{0}(\Omega)
\end{split}
\end{equation}
has uniquely a strong solution $u$ satisfying
\begin{equation}
k\|\nabla u\|^{2}_{C([0,T];L^{2}(\Omega))}+\|d_{t}u\|^{2}_{L^{2}((0,T);L^{2}(\Omega))}+\|u\|^{2}_{L^{2}((0,T);D(A))}\leq k\|\nabla u_{0}\|^{2}_{L^{2}(\Omega)}+\|g\|^{2}_{L^{2}((0,T);L^{2}(\Omega))}
\end{equation}
for any $g \in L^{2}((0,T);L^{2}(\Omega))$, $u_{0} \in H^{1}_{0}(\Omega)$, $T>0$.
It can be easily seen from (4.6) and the Banach fixed point theorem that there exists a (small) positive constant $T_{*}\leq T$ depending only on $\Omega$, $k$, $p$, $f\in L^{2}((0,T);L^{2}(\Omega))$ and $u_{0} \in H^{1}_{0}(\Omega)$ such that
\begin{equation}
\begin{split}
&d_{t}u+Au=-b(u)+f & \mathrm{in} \ L^{2}((0,T);L^{2}(\Omega)), \\
&u(0)=u_{0} & \mathrm{in} \ H^{1}_{0}(\Omega)
\end{split}
\end{equation}
has uniquely a strong solution $u$ satisfying $u \in L^{2}((0,T_{*});D(A))\cap C([0,T_{*}];H^{1}_{0}(\Omega))$, $d_{t}u \in L^{2}((0,T_{*});L^{2}(\Omega))$.
Moreover, a priori estimate for strong solutions of (4.2) is established as follows:
\begin{equation*}
kd_{t}(\|\nabla u(t)\|^{2}_{L^{2}(\Omega)})+\|d_{t}u(t)\|^{2}_{L^{2}(\Omega)}+k^{2}\lambda_{1}\|\nabla u(t)\|^{2}_{L^{2}(\Omega)}\leq 2\|u(t)\|^{2p}_{L^{2p}(\Omega)}+2\|f(t)\|^{2}_{L^{2}(\Omega)},
\end{equation*}
\begin{equation}
d_{t}(\|\nabla u(t)\|^{2}_{L^{2}(\Omega)})\leq -k\lambda_{1}\|\nabla u(t)\|^{2}_{L^{2}(\Omega)}+2k^{-1}C^{2p}\|\nabla u(t)\|^{2p}_{L^{2}(\Omega)}+2k^{-1}\|f(t)\|^{2}_{L^{2}(\Omega)}
\end{equation}
for any $t>0$, where $\lambda_{1}=\lambda_{1}(\Omega)>0$ is the first eigenvalue of $-\Delta$ with the zero Dirichlet boundary condition, $C=C(\Omega)$ is a positive constant.
Let us assume that
\begin{equation}
\|f\|^{2}_{L^{\infty}((0,\infty);L^{2}(\Omega))}\leq \left(\frac{k^{2}\lambda_{1}}{4C^{2p}}\right)^{p/(p-1)}, \ \|\nabla u_{0}\|^{2}_{L^{2}(\Omega)}\leq \left(\frac{k^{2}\lambda_{1}}{4C^{2p}}\right)^{1/(p-1)}.
\end{equation}
Then it follows from (4.8) that
\begin{equation}
\|\nabla u\|^{2}_{C_{b}([0,\infty);L^{2}(\Omega))}\leq \left(\frac{k^{2}\lambda_{1}}{4C^{2p}}\right)^{1/(p-1)}.
\end{equation}
By applying (4.10) to the unique solvability of (4.7), consequently, (4.2) has uniquely a strong solution provided that $f$ and $u_{0}$ satisfy (4.9).
\end{proof}
\begin{theorem}
Let $n=2, 3$, $1<p\leq n/(n-2)$, $0<\alpha\leq 1$, $f \in L^{\infty}((0,\infty);D(A^{\alpha}))$, $u_{0} \in H^{1}_{0}(\Omega)$, $t_{0}>0$.
Then there exists a positive constant $M(t_{0},\alpha)$ depending only on $\Omega$, $k$, $p$, $f$, $u_{0}$, $t_{0}$ and $\alpha$ such that $\|u\|_{C_{b}([t_{0},\infty);D(A))}\leq M(t_{0},\alpha)$ for any strong solution $u$ of $(4.2)$.
\end{theorem}
\begin{proof}
It is well known in \cite[Theorems 2.5.2 and 7.3.6]{Pazy} that $A$ is a sectorial operator in $L^{2}(\Omega)$, $\mathrm{Re}\sigma(A)>0$.
Since $u \in C_{b}([0,\infty);H^{1}_{0}(\Omega))$, it follows from \cite[Lemma 3.3.2]{Henry} that
\begin{equation}
u(t)=e^{-tA}u_{0}-\int^{t}_{0}e^{-(t-s)A}b(u)(s)ds+\int^{t}_{0}e^{-(t-s)A}f(s)ds
\end{equation}
for any $t\geq 0$.
In the case where $1/2<\beta<1$, it can be easily seen from (4.11) that there exists a positive constant $M(t_{0})$ depending only on $\Omega$, $k$, $p$, $\|f\|_{L^{\infty}((0,\infty);L^{2}(\Omega))}$, $\|u_{0}\|_{H^{1}(\Omega)}$,  $t_{0}$ and $\beta$ such that $\|u\|_{C_{b}([t_{0},\infty);D(A^{\beta}))}\leq M(t_{0})$.
Let $n/4<\beta<1$.
Then, since $D(A^{1/2})=H^{1}_{0}(\Omega)$, $(u,v)_{D(A^{1/2})}=(u,v)_{a}$,
\begin{equation*}
\begin{split}
\|b(u)(t)\|^{2}_{D(A^{1/2})}&\leq kp^{2}C^{2(p-1)}\|u(t)\|^{2(p-1)}_{L^{\infty}(\Omega)}\|\nabla u(t)\|^{2}_{L^{2}(\Omega)} \\
&\leq kp^{2}C^{2p}\|u(t)\|^{2p}_{D(A^{\beta})} \\
&\leq kp^{2}C^{2p}M(t_{0})^{2p},
\end{split}
\end{equation*}
\begin{equation}
\|b(u)(t)\|_{D(A^{1/2})}\leq k^{1/2}pC^{p}M(t_{0})^{p}
\end{equation}
for any $t\geq t_{0}$, where $C=C(\Omega)$ is a positive constant.
In the case where $\beta=1$, it follows from (4.11) that
\begin{equation}
\begin{split}
\|u(t)\|_{D(A)}\leq& C_{1/2,\lambda}(t-t_{0}/2)^{-1/2}e^{-\lambda(t-t_{0}/2)}\|u(t_{0}/2)\|_{D(A^{1/2})} \\
&+C_{1/2,\lambda}\int^{t}_{t_{0}/2}(t-s)^{-1/2}e^{-\lambda(t-s)}\|b(u)(s)\|_{D(A^{1/2})}ds \\
&+C_{1-\alpha,\lambda}\int^{t}_{t_{0}/2}(t-s)^{-1+\alpha}e^{-\lambda(t-s)}\|f(s)\|_{D(A^{\alpha})}ds
\end{split}
\end{equation}
for any $t\geq t_{0}>t_{0}/2$.
Therefore, we can conclude from (4.12), (4.13) that there exists a positive constant $M(t_{0},\alpha)$ depending only on $\Omega$, $k$, $p$, $\|f\|_{L^{\infty}((0,\infty);D(A^{\alpha}))}$, $\|u(t_{0}/2)\|_{H^{1}(\Omega)}$, $t_{0}$ and $\alpha$ such that $\|u\|_{C_{b}([t_{0},\infty);D(A))}\leq M(t_{0},\alpha)$.
\end{proof}
As for (4.2), we can obtain the following theorem:
\begin{theorem}
Let $n=2, 3$, $f, g \in L^{\infty}((0,\infty);L^{2}(\Omega))$, $u_{0}, v_{0} \in H^{1}_{0}(\Omega)$, $u \in \mathcal{S}(f,u_{0})$, $v \in \mathcal{S}(g,v_{0})$, and assume only $\mathrm{(H.3)}$.
Then there exists a positive constant $\delta_{4}$ depending only on $\Omega$, $A$, $B$, $f$, $g$, $M(f,u_{0},t_{0})$ and $M(g,v_{0},t_{0})$ such that if $0<d_{N}\leq\delta_{4}$, $u(x_{j},t)-v(x_{j},t)\rightarrow 0$ $(j=1, \cdots, N)$, $f(t)-g(t)\rightarrow 0$ in $L^{2}(\Omega)$ as $t\rightarrow \infty$, then $u(t)-v(t)\rightarrow 0$ in $H^{1}_{0}(\Omega)$ as $t\rightarrow \infty$.
\end{theorem}

\subsection{Navier-Stokes equations}
The initial-boundary value problem for the Navier-Stokes equations is described as follows:
\begin{equation}
\begin{split}
&\mathrm{div}u=0 & \mathrm{in} \ \Omega\times(0,\infty), \\
&\partial_{t}u-\mu\Delta u+(u\cdot\nabla)u=-\nabla p+f & \mathrm{in} \ \Omega\times(0,\infty), \\
&u|_{t=0}=u_{0} & \mathrm{in} \ \Omega, \\
&u|_{\partial\Omega}=0 & \mathrm{on} \ \partial\Omega\times(0,\infty),
\end{split}
\end{equation}
where $\mu>0$, $f$ is an external force field, $u_{0}$ is an initial data of $u$.

Let us introduce the solenoidal function spaces to utilize the strong formulation of (4.14).
$C^{\infty}_{0,\sigma}(\Omega):=\{u \in (C^{\infty}_{0}(\Omega))^{n} \ ; \ \mathrm{div}u=0\}$.
$L^{2}_{\sigma}(\Omega)$ is the completion of $C^{\infty}_{0,\sigma}(\Omega)$ in $(L^{2}(\Omega))^{n}$.
$L^{2}_{\sigma}(\Omega)$ is characterized as $L^{2}_{\sigma}(\Omega)=\{u \in (L^{2}(\Omega))^{n} \ ; \ \mathrm{div}u=0, \ \nu\cdot u|_{\partial\Omega}=0\}$, where $\nu$ is the outward normal vector on $\partial\Omega$.
It follows from the Helmholtz decomposition that $(L^{2}(\Omega))^{n}$ is decomposed into $(L^{2}(\Omega))^{n}=L^{2}_{\sigma}(\Omega)\oplus L^{2}_{\pi}(\Omega)$, where $L^{2}_{\pi}(\Omega):=\{\nabla p \ ; \ p \in H^{1}(\Omega) \}$.
Let $P_{2}$ be the orthogonal projection of $(L^{2}(\Omega))^{n}$ onto $L^{2}_{\sigma}(\Omega)$.
See, for example, \cite[Chapter 1]{Temam} on basic properties of the Helmholtz decomposition.

Let $H^{n}=L^{2}_{\sigma}(\Omega)$, $V^{n}=(H^{1}_{0}(\Omega))^{n}\cap L^{2}_{\sigma}(\Omega)$, $P=P_{2}$.
Then we can make use of the strong formulation to rewrite (4.14) by
\begin{equation}
\begin{split}
&d_{t}u+Au+B(u)=f & \mathrm{in} \ L^{2}((0,\infty);L^{2}_{\sigma}(\Omega)), \\
&u(0)=u_{0} & \mathrm{in} \ (H^{1}_{0}(\Omega))^{n}\cap L^{2}_{\sigma}(\Omega),
\end{split}
\end{equation}
where $Au=-P_{2}(\mu\Delta u)$, $B(u)=P_{2}(u\cdot\nabla)u$.
It follows from \cite[Lemma 3.3.7]{Temam} that $A$ satisfies (A.1)--(A.4).
It is clear from the following lemma that $Bu=B(u)$ satisfies $\mathrm{(B.1)}$, $\mathrm{(B.2)}$.
\begin{lemma}
Let $n=2,3$.
Then there exists a positive constant $C_{B}$ depending only on $\Omega$ such that we have the following inequality:
\begin{equation}
\|B(u)-B(v)\|_{(L^{2}(\Omega))^{n}}\leq C_{B}(\|u\|_{(H^{2}(\Omega))^{n}}+\|v\|_{(H^{2}(\Omega))^{n}})\|u-v\|_{(H^{1}(\Omega))^{n}}
\end{equation}
for any $u,v \in (H^{2}(\Omega))^{n}$.
\end{lemma}
\begin{proof}
It is obvious that
\begin{equation*}
\begin{split}
B(u)-B(v)&=P_{2}(u\cdot\nabla)u-P_{2}(v\cdot\nabla)v \\
&=P_{2}(u\cdot\nabla)(u-v)+P_{2}((u-v)\cdot\nabla)v
\end{split}
\end{equation*}
for any $u, v \in (H^{2}(\Omega))^{n}$.
Let us notice from the Sobolev embedding theorem that $(H^{1}(\Omega))^{n}$ and $(H^{2}(\Omega))^{n}$ are continuously embedded in $(L^{3}(\Omega))^{n}$ and $(L^{\infty}(\Omega))^{n}$ respectively.
Then we obtain that
\begin{equation*}
\begin{split}
\|P_{2}(u\cdot\nabla)(u-v)\|_{(L^{2}(\Omega))^{n}}&\leq \|(u\cdot\nabla)(u-v)\|_{(L^{2}(\Omega))^{n}} \\
&\leq\|u\|_{(L^{\infty}(\Omega))^{n}}\|\nabla(u-v)\|_{(L^{2}(\Omega))^{n}} \\
&\leq C_{1}\|u\|_{(H^{2}(\Omega))^{n}}\|u-v\|_{(H^{1}(\Omega))^{n}}
\end{split}
\end{equation*}
for any $u, v \in (H^{2}(\Omega))^{n}$, where $C_{1}=C_{1}(\Omega)$ is a positive constant.
It follows from the H\"{o}lder inequality and the same argument as above that
\begin{equation*}
\begin{split}
\|P_{2}((u-v)\cdot\nabla)v\|_{(L^{2}(\Omega))^{n}}&\leq \|((u-v)\cdot\nabla)v\|_{(L^{2}(\Omega))^{n}} \\
&\leq\|u-v\|_{(L^{6}(\Omega))^{n}}\|\nabla v\|_{(L^{3}(\Omega))^{n}} \\
&\leq C_{2}\|v\|_{(H^{2}(\Omega))^{n}}\|u-v\|_{(H^{1}(\Omega))^{n}}
\end{split}
\end{equation*}
for any $u, v \in (H^{2}(\Omega))^{n}$, where $C_{2}=C_{2}(\Omega)$ is a positive constant.
The above two inequalities lead clearly to (4.16).
\end{proof}
It is well known in \cite{Foias 2} that (H.3), (H.4) hold for (4.15) in the case where $n=2$, (H.3) implies (H.4) in the case where $n=3$.

\section*{Acknowledgment}
I am grateful to Professor Masahiro Yamamoto (Graduate School of Mathematical Sciences, The University of Tokyo) for his valuable advice and constant encouragement.


\end{document}